\documentclass[11pt]{article}
\usepackage{amsmath,amsthm, amssymb, amsfonts,accents,nccmath}
\usepackage{enumitem}
\usepackage{array}
\usepackage{lipsum}
\usepackage{mathtools}
\usepackage[toc,page]{appendix}
\usepackage[english]{babel}
\usepackage{dsfont}
\usepackage{booktabs}
\usepackage[linesnumbered,commentsnumbered,ruled,vlined]{algorithm2e}
\usepackage{tikz}
\usepackage{tikz-network}
\usepackage[utf8]{inputenc}
\usepackage[english]{babel}

\DeclarePairedDelimiter\floor{\lfloor}{\rfloor}

\usepackage{caption}
\usepackage{subcaption}

\newcommand\restr[2]{{% we make the whole thing an ordinary symbol
  \left.\kern-\nulldelimiterspace % automatically resize the bar with \right
  #1 % the function
  \vphantom{\big|} % pretend it's a little taller at normal size
  \right|_{#2} % this is the delimiter
  }}

\usepackage{hyperref}
\hypersetup{
    colorlinks=true,
    linkcolor=blue,
    filecolor=darkmagenta,      
    urlcolor=cyan,
    citecolor=red
} 
\makeatletter
\def\BState{\State\hskip-\ALG@thistlm}
\makeatother
\numberwithin{equation}{section}

\newtheorem{theorem}{Theorem}[section]

\newtheorem{lem}[theorem]{Lemma}

\newtheorem{rem}[theorem]{Remark}

\newtheorem{conject}[theorem]{Conjecture}

\newcommand{\ds}{\displaystyle}

\title{Proof of a Conjecture on Wiener Index and Eccentricity of a graph due to edge contraction}
\author{Joyentanuj Das$^*$   \quad and \quad Ritabrata Jana\footnote{School of Mathematics, IISER Thiruvananthapuram, Maruthamala P.O., Vithura, 
Thiruvananthapuram,\newline \indent   Kerala - 695551, India. 
\newline Emails:
joyentanuj16@iisertvm.ac.in,  ritabrata20@iisertvm.ac.in}}

\date{}

\begin{document}

\maketitle

\begin{abstract}
For a connected graph $G$, the Wiener index, denoted by $W(G)$, is the sum of the distance of all pairs of distinct vertices and the eccentricity, denoted by $\varepsilon(G)$, is the sum of the eccentricity of individual vertices. In \cite{Kc}, the authors posed a conjecture which states that given a graph $G$ with at least three vertices, the difference between $W(G)$ and $\varepsilon(G)$ decreases when an edge is contracted and proved that the conjecture is true when $e$ is a bridge. In this manuscript, we confirm that the conjecture is true for any connected graph $G$ with at least three vertices irrespective of the nature of the edge chosen.
\end{abstract}



\section{Introduction}
Let $G=(V(G),E(G))$  be a finite, simple, connected graph with $V(G)$ as the set of vertices and $E(G)$ as the set of edges in $G$. We write $u\sim v$ to indicate that the vertices $u$ and $v$ are adjacent in $G$. \textcolor{black}{ We denote the complete graph on $n$ vertices by $K_n$ and the path graph on $n$ vertices by $P_n$. A vertex $u$ is said to be a neighbour of a vertex $v$ if $u \sim v$. The collections of all such neighbours of $v$ in $G$ is denoted by $N_G(v)$.}

For a given edge $e$, we write $G.e$ to denote the graph obtained from $G$ by contracting the edge $e$. More precisely, if $e$ is the edge between two vertices $x$ and $y$ in $G$ then, the vertices $x$ and $y$ are merged contracting the edge $e$ in $G.e$ and we rename the vertex as $\alpha$. Note that, due to this graph transformation we have $N_{G.e}(\alpha) = N_G(x) \cup N_G(y)$. For the vertices $x$ and $y$ we will use $x( \text{or } y)$ to denote inclusive or, \textit{i.e.} $x$ or $y$ or both.

A $uv$-path in $G$ is a path in $G$ whose end vertices are $u$ and $v$. Let $u \in V(G)$ and $P$ be a path in $G$. We say that a path $P$ uses a vertex $u$ if $u \in V(P)$. Similarly, by saying that a path $P$ uses an edge $e$ we mean that $e \in E(P)$.

A connected graph $G$ is a metric space with respect to the metric $d,$ where $d_G(u,v)$ equals the length of a minimal $uv$-path. We set $d_G(u,u)=0$ for every vertex $u$ in $G$. A $uv$-path in $G$ is said to be of minimal length if the length of the path is equal to $d_G(u,v)$. The Wiener index of a graph $G$, denoted by $W(G)$ is defined as $$W(G) = \sum_{\{u,v\} \subset V(G)} d_G(u,v) = \frac{1}{2} \sum_{u \in V(G)} \sum_{v \in V(G)} d_G(u,v).$$ The Wiener index is the oldest topological index studied in mathematical chemistry and is one of the most studied among such indices (for surveys one may refer to \cite{Do1, Kn2}) and is an active area of research (for details see \cite{Ca, Kc, Kc2, Do, Kn1, Pe} and the references there in).

If $u \in V(G)$ then the eccentricity $\varepsilon_G(u)$ is the distance from $u$ to a farthest vertex from $u$. A vertex $v$ is said to be an eccentric vertex of $u$ if $d_G(u,v) = \varepsilon_G(u)$. The eccentricity of a graph $G$ is $$\varepsilon(G) = \sum_{u \in V(G)} \varepsilon_G(u)$$ which is also known as total  eccentricity of a graph. The radius $\textup{rad}(G)$ of $G$ and the diameter $\textup{diam}(G)$ of $G$ are the minimum and maximum eccentricity, respectively.  For studies related to total eccentricity of a graph and average eccentricity of a graph one may refer to \cite{Da}-\cite{Kc1} and \cite{De}. In \cite{Ca}, the Wiener index and the average eccentricity has been studied on strong products of graphs.  In \cite{Kc}, the authors studied the relation between Wiener index and eccentricity for certain classes of graphs and posed the following conjecture:

\begin{conject}\label{Con1}
If $e$ is an edge of a graph $G$ with number of vertices at least $3$, then $$W(G.e) - \varepsilon(G.e) \le W(G) - \varepsilon(G).$$
\end{conject}
%\begin{conject}\label{Con2}
%If $G$ is a graph with $n$ vertices and $\textup{rad}(G) \ge 4$, then $$W(G) - \varepsilon(G) \le \floor*{\frac{1}{6}n^3 - \frac{3}{4}n^2 + \frac{1}{3}n + \frac{1}{4}}$$ with equality holding if and only if $G$ is a path.
%\end{conject}

An edge in a connected graph is a bridge if and only if removing it disconnects the graph. In \cite{Kc}, the authors proved the following theorem when the edge is a bridge as partial support for the Conjecture~\ref{Con1}. 

\begin{theorem}
If $e$ is a bridge of a graph $G$ with at least $3$ vertices, then $$W(G.e) - \varepsilon(G.e) \le W(G) - \varepsilon(G).$$
\end{theorem}

In this manuscript, we prove that the Conjecture~\ref{Con1} is true.

\section{Proof of Conjecture~\ref{Con1}}
Let $e$ be an edge in $G$ between the vertices $x$ and $y$ and the graph obtained from $G$ by contracting the edge $e$ is denoted by $G.e$. Let $\alpha$ be the vertex in $G.e$, which is formed by merging the vertices $x$ and $y$.

\begin{lem}\label{Lem:D1}
Let $u,v$ be two vertices in $G.e$ which is different from the vertex $\alpha$ then 
$$
d_{G.e} (u,v) = 
\begin{cases}
d_G(u,v) \textup{ or} \\
d_G(u,v) - 1.
\end{cases}
$$
\end{lem}

\begin{proof}
Let $P$ be a $uv$-path in $G$ then, we have two possibilities; either $P$ uses the edge $e$ or it does not. Suppose all $uv$-paths of length $d_G(u,v)$ does not use the edge $e$ then, the paths remain preserved in $G.e$ and we have $d_{G.e} (u,v) = d_G(u,v)$. If there exists a $uv$-path in $G$ of length $d_G(u,v)$ that uses the edge $e$ then, the length of the $uv$-path in $G.e$ is decreased by $1$ and we have $d_{G.e} (u,v) = d_G(u,v) - 1$.
\end{proof}

\begin{lem}\label{Lem:D2}
Let $u \neq \alpha$ be a vertex in $G.e$ then
\begin{equation*}
d_{G.e} (u,\alpha) = 
\begin{cases}
d_G(u,x) \textup{ or}\\
d_G(u,x) - 1.
\end{cases}
\end{equation*}
\end{lem}

\begin{proof}
Let $P$ be a path of minimal length in $G.e$ between $u$ and $\alpha$. Let $u_n$ be the vertex on the path $P$ which is adjacent to $\alpha$. Then in $G$, the vertex $u_n$ is either adjacent to $x$ or $y$ or both. Next, we consider the following cases to complete the proof.

\underline{\textbf{Case 1:}} Suppose for all minimal paths of length $d_{G.e} (u,\alpha)$, the vertex $u_n$ adjacent to $\alpha$ is not adjacent to $x$ in $G$. Then $u_n \sim y$ and the path $P_1 = u \sim \cdots \sim u_n \sim y \sim x$ from $u$ to $x$ is a path of minimal length in $G$ and hence we have $d_G(u,x) = d_{G.e} (u,\alpha) +1$.

\underline{\textbf{Case 2:}} Suppose there exists a minimal path of length $d_{G.e} (u,\alpha)$ such that the vertex $u_n$ is adjacent to $\alpha$ in $G.e$ and is adjacent to $x$ in $G$. Then it follows that $d_G(u,x) = d_{G.e} (u,\alpha)$.
\end{proof}

\begin{lem}\label{Lem:E1}
Let $x ($or $y)$ be the eccentric vertex(or vertices) of $u$ in $G$. Then, $\alpha$ is the eccentric vertex of $u$ in $G.e$.
\end{lem}

\begin{proof}
Suppose on the contrary we assume that $\alpha$ is not the eccentric vertex of $u$ and $w \neq \alpha$ is an eccentric vertex of $u$ in $G.e$. Then we have
\begin{equation}\label{eqn:E1_e1}
d_{G.e} (u,\alpha) \le d_{G.e} (u,w) - 1 \text{ and }
\end{equation}
\begin{equation}\label{eqn:E1_e2}
d_G(u,w) \le d_G(u,x) - 1.
\end{equation} Thus, combining Eqns.~\eqref{eqn:E1_e1} and \eqref{eqn:E1_e2} we have $$d_{G.e} (u,\alpha) +1 \le d_{G.e} (u,w) \le d_G(u,w) \le d_G(u,x) - 1,$$ which implies that $d_{G.e} (u,\alpha) \le d_G(u,x) - 2$, but this is a contradiction by Lemma~\ref{Lem:D2} and the result follows.
\end{proof}

\begin{lem}\label{Lem:E2}
If there is an eccentric vertex of $u$ in $G$ other than $x$ and $y$ then there exist an eccentric vertex of $u$ that is common in both $G$ and $G.e$.
\end{lem}

\begin{proof}
We prove the lemma by considering the following two cases:

\underline{\textbf{Case 1:}} Let $w_1,w_2,\cdots,w_k$ be the eccentric vertices of $u$ in $G$, such that none of the $w_i$'s are equal to $x$ or $y$. Suppose on the contrary, we assume that $w \neq w_i$ for $1 \le i \le k$ is an eccentric vertex of $u$ in $G.e$. Then, the following holds
\begin{equation}\label{eqn:E2_e1}
d_{G.e} (u,w_i) \le d_{G.e} (u,w) - 1 \text{ for } 1 \le i \le k  \text{ and }
\end{equation}
\begin{equation}\label{eqn:E2_e2}
d_G(u,w) \le d_G(u,w_i) - 1 \text{ for } 1 \le i \le k .
\end{equation} Combining Eqns.~\eqref{eqn:E2_e1} and \eqref{eqn:E2_e2} we have $$d_{G.e} (u,w_i) +1 \le d_{G.e} (u,w) \le d_G(u,w) \le d_G(u,w_i) - 1,$$ which implies that $d_{G.e} (u,w_i) \le d_G(u,w_i) - 2$, but this is a contradiction by Lemma~\ref{Lem:D1} and hence $w = w_i$ for some $1 \le i \le k$.

\underline{\textbf{Case 2:}}  Let $w_1,w_2,\cdots,w_k$ be the eccentric vertices of $u$ in $G$ other than $x$ and $y$. Observe that for $1 \le i \le k$, none of the minimal $uw_i$-paths in $G$ use the edge $e$. If a $uw_i$-path uses the edge $e$ then we have either $u \sim \cdots \sim x \sim y \sim \cdots \sim w_i$ or $u \sim \cdots \sim y \sim x \sim \cdots \sim w_i$ but in any of the cases $d_G(u,w_i) > d_G(u,x)$, which is a contradiction. Since the $uw_i$-path of minimal length does not use the edge $e$ in $G$ the same is preserved in $G.e$, \textit{i.e.} $d_G(u,w_i) = d_{G.e}(u,w_i)$. If there are no other eccentric vertices $w$ other than $\alpha$ in $G.e$ then, $d_{G.e} (u,w) < d_{G.e} (u,\alpha)$ for all $w \in V(G.e)\setminus \{\alpha\}$. But all the $uw$-paths of minimal length in $G.e$ are preserved in $G$ which implies that $x$ or $y$ are the only eccentric vertices of $u$ in $G$, which is a contradiction. Now suppose $w \neq w_i$ for $1 \le i \le k$, be an eccentric vertex of $u$ in $G.e$ then by similar arguments as in the previous case we arrive at a contradiction and hence the result follows.

\end{proof}

Now we are ready to prove the Conjecture~\ref{Con1}.

\begin{proof}[Proof of Conjecture~\ref{Con1}]
Let $G$ be a connected graph on $n$ vertices. If $n = 3$ then, $G$ is either the complete graph $K_3$ or the path graph $P_3$ of length $2$. In either of the cases the resulting graph after contraction of an edge  will lead to a single edge \textit{i.e.} $K_2$. It is easy to see that $W(K_2) - \varepsilon(K_2) \le W(G) - \varepsilon(G)$. Thus, the result is true when $n = 3$. Now we consider the case when $n \ge 4$. The difference between the Wiener index and eccentricity of $G$ and $G.e$ can be expressed as
\begin{align*}
\begin{split}
W(G) - \varepsilon(G) &= \sum_{u \in V(G)} \left(\frac{1}{2} \sum_{v \in V(G)} d_G(u,v) - \max_{v \in V(G)} d_G(u,v)  \right).\\
W(G.e) - \varepsilon(G.e) &= \sum_{u \in V(G.e)} \left(\frac{1}{2} \sum_{v \in V(G.e)} d_{G.e}(u,v) - \max_{v \in V(G.e)} d_{G.e}(u,v)  \right).
\end{split}
\end{align*}
Let $\widetilde{V}$ denote the set of vertices that are common in both $G$ and $G.e$, \textit{i.e.} $ V(G) = \widetilde{V} \cup \{x,y\}$ and $V(G.e) = \widetilde{V} \cup \{\alpha\}$. We complete the prove by showing that for all $u \in \widetilde{V}$
\begin{equation}\label{eqn:Conj0}
\frac{1}{2} \sum_{v \in V(G)} d_G(u,v) - \max_{v \in V(G)} d_G(u,v) \ge \frac{1}{2} \sum_{v \in V(G.e)} d_{G.e}(u,v) - \max_{v \in V(G.e)} d_{G.e}(u,v) 
\end{equation}  and 
\begin{equation}\label{eqn:Conj1}
\begin{split}
\frac{1}{2} \sum_{v \in V(G)} d_G(x,v) - \max_{v \in V(G)} d_G(x,v) + \frac{1}{2} \sum_{v \in V(G)} d_G(y,v) - \max_{v \in V(G)} d_G(y,v) \\
\ge \frac{1}{2} \sum_{v \in V(G.e)} d_{G.e}(\alpha,v) - \max_{v \in V(G.e)} d_{G.e}(\alpha,v).
\end{split}
\end{equation}

To prove Eqns.	~\eqref{eqn:Conj0} and \eqref{eqn:Conj1} we consider the following cases.

\underline{\textbf{Case 1:}} Let $u \in \widetilde{V}$.

\underline{\textbf{Subcase 1.1:}} Let $x ( \text{or } y)$ be the eccentric vertex (or vertices) of $u$ in $G$. Thus, by Lemma~\ref{Lem:E1} $\alpha$ is the eccentric vertex of $u$ in $G.e$. To prove Eqn.~\eqref{eqn:Conj0} it is enough to show that
\begin{equation*}
\frac{1}{2} \sum_{v \in V(G)} d_G(u,v) - \max_{v \in V(G)} d_G(u,v) - \frac{1}{2} \sum_{v \in V(G.e)} d_{G.e}(u,v) + \max_{v \in V(G.e)} d_{G.e}(u,v) \ge 0.
\end{equation*} Simplifying the left side of the inequality we have,

\begin{align*}
&\frac{1}{2} \sum_{v \in V(G)} d_G(u,v) - \max_{v \in V(G)} d_G(u,v) - \frac{1}{2} \sum_{v \in V(G.e)} d_{G.e}(u,v) + \max_{v \in V(G.e)} d_{G.e}(u,v)\\
&= \frac{1}{2} \sum_{v \in \widetilde{V}} d_G(u,v) +\frac{1}{2} d_G(u,x) + \frac{1}{2} d_G(u,y) - \max_{v \in V(G)} d_G(u,v) \\
& \quad \quad - \frac{1}{2} \sum_{v \in \widetilde{V}} d_{G.e}(u,v) - \frac{1}{2} d_{G.e}(u,\alpha) + \max_{v \in V(G.e)} d_{G.e}(u,v)\\
&= \frac{1}{2} \sum_{v \in \widetilde{V}} d_G(u,v) -\frac{1}{2} d_G(u,x) + \frac{1}{2} d_G(u,y) - \frac{1}{2} \sum_{v \in \widetilde{V}} d_{G.e}(u,v) + \frac{1}{2} d_{G.e}(u,\alpha) \\
&= \frac{1}{2} \left( \sum_{v \in \widetilde{V}} d_G(u,v) -\sum_{v \in \widetilde{V}} d_{G.e}(u,v)  \right) + \frac{1}{2} \left( d_G(u,y) +  d_{G.e}(u,\alpha) -d_G(u,x) \right).
\end{align*} Finally, using Lemmas~\ref{Lem:D1} and~\ref{Lem:D2} and using the fact that $u \neq y$ the result follows. Note that we have used the fact that $x$ is an eccentric vertex of $u$. Similar calculations will follow if $y$ is an eccentric vertex of $u$.

\underline{\textbf{Subcase 1.2:}} Let $w$ be an eccentric vertex of $u$ in $G$ other than $x$ and $y$. Without loss of generality using Lemma~\ref{Lem:E2}, we can choose $w \in \widetilde{V}$ such that $w$ is an eccentric vertex of $u$ in both $G$ and $G.e$. To prove Eqn.~\eqref{eqn:Conj0} it is enough to show that
\begin{equation*}
\frac{1}{2} \sum_{v \in V(G)} d_G(u,v) - \max_{v \in V(G)} d_G(u,v) - \frac{1}{2} \sum_{v \in V(G.e)} d_{G.e}(u,v) + \max_{v \in V(G.e)} d_{G.e}(u,v) \ge 0.
\end{equation*} Simplifying the left side of the inequality we have,
\begin{align*}
&\frac{1}{2} \sum_{v \in V(G)} d_G(u,v) - \max_{v \in V(G)} d_G(u,v) - \frac{1}{2} \sum_{v \in V(G.e)} d_{G.e}(u,v) + \max_{v \in V(G.e)} d_{G.e}(u,v)\\
&= \frac{1}{2} \sum_{v \in \widetilde{V}} d_G(u,v) +\frac{1}{2} d_G(u,x) + \frac{1}{2} d_G(u,y) - \max_{v \in V(G)} d_G(u,v) \\
& \qquad \qquad - \frac{1}{2} \sum_{v \in \widetilde{V}} d_{G.e}(u,v) - \frac{1}{2} d_{G.e}(u,\alpha) + \max_{v \in V(G.e)} d_{G.e}(u,v)\\
&= \frac{1}{2} \sum_{v \in \widetilde{V}\setminus\{w\}} d_G(u,v) +\frac{1}{2} d_G(u,x) + \frac{1}{2} d_G(u,y) -  \frac{1}{2} d_G(u,w)\\
& \qquad \qquad - \frac{1}{2} \sum_{v \in \widetilde{V}\setminus\{w\}} d_{G.e}(u,v) - \frac{1}{2} d_{G.e}(u,\alpha) + \frac{1}{2} d_{G.e}(u,w)\\
&= \frac{1}{2} \left( \sum_{v \in \widetilde{V}\setminus\{w\}} d_G(u,v) -  \sum_{v \in \widetilde{V}\setminus\{w\}} d_{G.e}(u,v)\right)\\
& \qquad \qquad +\frac{1}{2} \left( d_G(u,x) +  d_G(u,y) -   d_G(u,w) -  d_{G.e}(u,\alpha) +  d_{G.e}(u,w) \right).
\end{align*}
Since $\ds \sum_{v \in \widetilde{V}\setminus\{w\}} d_G(u,v) -  \sum_{v \in \widetilde{V}\setminus\{w\}} d_{G.e}(u,v) \ge 0$ follows from Lemma~\ref{Lem:D1}, it only remains to show that $d_G(u,x) +  d_G(u,y) -  d_{G.e}(u,\alpha) \ge  d_G(u,w) - d_{G.e}(u,w).$ Note that, using Lemma~\ref{Lem:D1} we have $d_G(u,w) - d_{G.e}(u,w)$ is either $0$ or $1$. Similarly, by Lemma~\ref{Lem:D2} $d_G(u,x) - d_{G.e}(u,\alpha)$ is either $0$ or $1$. Combining, we have  $d_G(u,x) +  d_G(u,y) -  d_{G.e}(u,\alpha) \ge 1$ since $u \neq y$ and hence the result follows.

\underline{\textbf{Case 2:}} In this case we prove the inequality~\eqref{eqn:Conj1}. Let $w_1$ and $w_2$ be eccentric vertices of $x$ and $y$, respectively. Without loss of generality, we assume that $w_1$ is an eccentric vertex of $\alpha$. Then we have the following
\begin{align*}
&\frac{1}{2} \sum_{v \in V(G)} d_G(x,v) - \max_{v \in V(G)} d_G(x,v) + \frac{1}{2} \sum_{v \in V(G)} d_G(y,v) - \max_{v \in V(G)} d_G(y,v) \\
& \quad \quad - \frac{1}{2} \sum_{v \in V(G.e)} d_{G.e}(\alpha,v) + \max_{v \in V(G.e)} d_{G.e}(\alpha,v)\\
&= \frac{1}{2} \sum_{v \in V(G)\setminus\{w_1\}} d_G(x,v) - \frac{1}{2} d_G(x,w_1) + \frac{1}{2} \sum_{v \in V(G)\setminus\{w_2\}} d_G(y,v) - \frac{1}{2} d_G(y,w_2) \\
& \quad \quad - \frac{1}{2} \sum_{v \in V(G.e)\setminus\{w_1\}} d_{G.e}(\alpha,v) + \frac{1}{2} d_{G.e}(\alpha,w_1)\\
&= \frac{1}{2} \left(\sum_{v \in V(G)\setminus\{w_1,y\}} d_G(x,v) - \sum_{v \in V(G.e)\setminus\{w_1\}} d_{G.e}(\alpha,v) \right)\\
&\quad \quad +\frac{1}{2} \left(  \sum_{v \in V(G)\setminus\{w_2\}} d_G(y,v)+d_G(x,y)  - d_G(x,w_1) - d_G(y,w_2) + d_{G.e}(\alpha,w_1) \right).
\end{align*}

%\textit{Claim $2$:} $\ds \sum_{v \in V(G)\setminus\{w_2\}} d_G(y,v)+d_G(x,y) - d_G(y,w_2) \ge d_G(x,w_1) - d_{G.e}(\alpha,w_1)$

Since the graph $G$ is connected there exists a vertex $w_3$ on the $yw_2$-path of minimal length such that $d_G(y,w_2) = d_G(y,w_3) + 1$. From Lemma~\ref{Lem:D2} we have $d_G(x,w_1) - d_{G.e}(\alpha,w_1)$ is at most $1$. Thus, to show the fact that
\begin{equation}\label{eqn:claim2}
\sum_{v \in V(G)\setminus\{w_2\}} d_G(y,v)+d_G(x,y) - d_G(y,w_2) \ge d_G(x,w_1) - d_{G.e}(\alpha,w_1)
\end{equation} it is enough to prove that $$\sum_{v \in V(G)\setminus\{w_2,w_3\}} d_G(y,v) +d_G(x,y) \ge 2.$$ But this is always true since $G$ has at least four vertices. Finally, using Lemma~\ref{Lem:D2} and Eqn.~\ref{eqn:claim2} we have,
\begin{align*}\label{eqn:Conj3}
\begin{split}
\frac{1}{2} \sum_{v \in V(G)} d_G(x,v) - \max_{v \in V(G)} d_G(x,v) + \frac{1}{2} \sum_{v \in V(G)} d_G(y,v) - \max_{v \in V(G)} d_G(y,v)\\
\ge \frac{1}{2} \sum_{v \in V(G.e)} d_{G.e}(\alpha,v) - \max_{v \in V(G.e)} d_{G.e}(\alpha,v).
\end{split}
\end{align*}
Thus, combining all the above cases and using the inequalities~\eqref{eqn:Conj0} and \eqref{eqn:Conj1} we have $$W(G.e) - \varepsilon(G.e) \le W(G) - \varepsilon(G).$$
\end{proof}

\begin{rem}
In \cite{Kc}, the authors posed a second conjecture stating that the difference between the Wiener index of a
graph and its eccentricity is largest possible on paths. If $G$ be a graph of order $n$ with $\textup{rad}(G) \ge 4$, then
$$W(G) - \varepsilon(G) \le \floor*{\frac{1}{6}n^3 - \frac{3}{4}n^2 + \frac{1}{3}n + \frac{1}{4}}$$ with equality holding if and only if $G$ is a path. The Conjecture is still open.
\end{rem}

\noindent{ \textbf{\Large Acknowledgements}}\\
We would like to thank Dr. Sumit Mohanty for his comments, suggestions and help in improving the presentation of the manuscript.

\end{document}